\documentclass[12pt,a4paper]{article}

\usepackage{amsmath, amsfonts, amssymb}
\usepackage{color}
  \evensidemargin
\oddsidemargin
\def\be#1\ee{\begin{equation}#1\end{equation}}

\newtheorem{thm}{Theorem}

\newtheorem{prop}[thm]{Proposition}
\newtheorem{cor}[thm]{Corollary}
\newtheorem{example}[thm]{Example}
\newtheorem{rem}[thm]{Remark}

\def\R{\mathbb{R}}
\def\N{{\mathbb N}}
\newenvironment{proof}[1][] {\noindent {\bf Proof#1:} }{\hspace*{\fill}$\square$\medskip\par}

\def\C{{\mathbb C}}

\def\D{{\mathbb D}}

\def\TI{I^{(\textsc{t})}}
\def\Ga{\Gamma}
\def\TGa{\Gamma^{(\textsc{t})}}

\def\CCC{\C^+}
\def\DDD{\D^+}

\def\eqg{\stackrel{(\textsc{t})}{\equiv}}

\title{On the completion of Skorokhod space}
\author{Mikhail Lifshits and Vladislav Vysotsky}

\begin{document}
\maketitle

\abstract{We consider the classical Skorokhod space $\D[0,1]$ and the space of continuous functions $\C[0,1]$ equipped
with the standard Skorokhod distance $\rho$.

It is well known that neither $(\D[0,1],\rho)$ nor $(\C[0,1],\rho)$ is complete.
We provide an explicit description of the corresponding completions. The elements of these completions
can be regarded as usual functions on $[0,1]$ except for a countable number of instants
where their values vary ``instantly".

\medskip



}

\section{Introduction}

We consider the classical Skorokhod space $\D:=\D[0,1]$  (the space of all {\it c\`adl\`ag}
functions\footnote{A function is c\`adl\`ag if it is right-continuous and has left limit at
every point.}
on $[0,1]$) and the space of continuous
functions $\C:=\C[0,1]$, both equipped with the standard simple Skorokhod distance
\[
  \rho(f,g):= \inf_{\gamma}\left[ \sup_{s\in[0,1]} |f(s)-g(\gamma(s))|+ \sup_{s\in[0,1]} |s-\gamma(s)|
  \right]
\]
where the infimum is taken over all continuous strictly increasing maps $\gamma$ acting from $[0,1]$
onto itself.

Recall that $\C$ is a closed linear subspace in $(\D,\rho)$. Indeed, if $g_n\in \C$, $f\in \D$, and
$\rho(f,g_n)\to 0$, as $n\to\infty$, then there exist continuous functions $\gamma_n$ such that
\[
   \lim_{n\to\infty} \sup_{s\in[0,1]} |f(s)-g_n(\gamma_n(s))|
   \le  \lim_{n\to\infty} \left[ \rho(f,g_n) + n^{-1}\right] =0.
\]
Therefore, $f$ is a uniform limit of continuous functions $g_n\circ \gamma_n$, hence, $f\in\C$.

The space $\D$ equipped with the Skorokhod topology generated by $\rho$ is a de facto standard
framework in the theory of stochastic processes with jumps, such as L\'evy processes, random walks,
etc.\ (see Skorokhod~\cite{Skorokhod1956,Skorokhod1957}).

It is well known that neither $(\C,\rho)$ nor $(\D,\rho)$ is complete (see Example~\ref{example:triang}
below). Usually this issue is solved by introducing the other, more delicate, distance generating the same
topology (see e.g.\ Billingsley~\cite[Section 14]{Bil}). The other option is to enlarge the spaces $\C$
and $\D$ completing them with respect to $\rho$. The goal of this paper is to describe the completions explicitly.

We mention a related study by Whitt~\cite[Section 15]{Whitt}, aimed to enlarge the space $\D$
in order to let more sequences converge (in the Skorokhod topologies $M_1$ and $M_2$).
Whitt's motivation was in queueing theory.

The completions of the spaces $(\C,\rho)$ and $(\D,\rho)$ appeared, respectively,
in the important functional large deviations principles of Mogulskii for random
walks~\cite{Mogulskii1976} %(Part I of Theorem 2)
and L{\'e}vy processes~\cite{Mogulskii1993}. %(Theorems 2.5 and~2.7).
Both papers used the completed spaces
without caring about their nature. However, it  is not quite clear how to work in such abstract setting,
e.g.\ how to calculate the large deviations rate function on the entire completed space.
\medskip

We finish this introduction with a simple but representative and instructive example.

\begin{example} \label{example:triang}
{\rm
Consider a triangular function $g(x):=(1-|x|)_+$ and introduce a family of functions
$\{g_\theta \}_{\theta >2}$ in $\C$ given by
\[
   g_\theta(s):= g\left(\theta(s-\tfrac 12)\right), \qquad 0\le s\le 1.
\]
By using a piecewise linear variable change it is straightforward to show that
\[
   \rho\left(g_{\theta_1},g_{\theta_2}\right)\le \left|\frac{1}{\theta_1}- \frac{1}{\theta_2}\right|
   \to 0, \qquad \textrm{as } \theta_1,\theta_2 \to \infty.
\]
Hence, $\{g_\theta\}_\theta$ has  the Cauchy property, as $\theta\to\infty$. However, it is clear that
there is no limit  $\lim_{\theta\to \infty} g_\theta$ in $(\D,\rho)$ because
\[
   \lim_{\theta\to\infty} g_{\theta}(s) = \begin{cases}1,& s=\tfrac 12,\\
   0,& s\not= \tfrac 12.
   \end{cases}
\]

Notice that every function $g_\theta$ takes all values between zero and one going first up, then down.
Informally, one may say that the limiting  ``function" should take all these values in the same order
(up and down) at the single time instant $s_*=\tfrac 12$. We call this paradoxical behaviour an
``{\it instanton}" and formalize it in the following sections.
}
\end{example}

\section{Notation and construction}

\subsection{Turbofunctions}
Let us denote $I:=[0,1]$  and $\TI:=[0,1]$, stressing the following subtle but important difference. We consider
$I$ as a usual time interval equipped with the standard distance
while $\TI$ is considered as a {\it topological} space, without any distance on it.

Distinguishing between $I$ and $\TI$ is an important part of our notation system. These objects coincide
as sets but they are equipped with different structures and have different meaning.
The set $I$ denotes the usual time interval (clock) and one may measure the distance
between the time instants, see e.g. the second term in the definition of  Skorokhod distance $\rho$.
On the contrary, $\TI$ is not considered as a true time interval, and one never measures the distance between its elements.
We use $\TI$ just as a technical tool for parametrisation, as one does in the definition of a path in general
topology.
We will be rather consequent in distinguishing notation for the elements of two sets,
using $t\in\TI$ and $s\in I$.
\medskip

Use the standard notation $\C[\TI]$ and $\D[\TI]$ for the spaces of continuous and c\`adl\`ag functions on $\TI$, respectively.
%; recall that $\C$ and $\D$ denotes the respective space on $I$.

%%%The use of $\TI$ instead of $I$ suggests a possibility to equip it with various %(pseudo-)metrics, which we will do below.
%We can think that $I$ is naturally equipped with a ``physical'' (or ``real'') time scale, with the time measured using
%%the %standard distance, while on $\TI$ we have only the ``topological'' time which has no scale but only direction
%%defined by %the natural total order.

We denote by $\Ga$ the class of all increasing homeomorphisms of $I$, i.e. strictly
increasing continuous functions acting from $I$ onto $I$. Similarly, $\TGa$ denotes
the class of all increasing homeomorphisms of $\TI$.
Denote by $\Sigma$ the class of all continuous {\it non-decreasing} mappings
acting from $\TI$ onto~$I$.

The {\it extended Skorokhod space} is defined as a Cartesian product $\DDD:=\D[\TI]\times \Sigma$.
Its elements, the pairs
\[
   F^{\sigma}:=(F,\sigma), \qquad F\in \D[\TI], \ \sigma\in \Sigma,
\]
will be called {\it turbofunctions}. Similarly, put $\CCC:=\C[\TI]\times \Sigma$.

If $\sigma\in\Sigma$ is {\it strictly} monotone, hence invertible, then every turbofunction $F^\sigma$
may be {\it visualized} as a usual function $\widehat{F^\sigma}\in\D[I]$ defined by
\begin{equation} \label{eqn:visual}
  \widehat{F^\sigma}(s):= F(\sigma^{-1}(s)), \qquad s\in I.
\end{equation}

If $\sigma$ is not strictly monotone, such simple visualization is not possible. In this case,
there exists a finite or countable number of $s\in I$
such that $\sigma^{-1}(s)$ is a non-degenerate interval in $\TI$.
%% $[\sigma^{-1}(s-), \sigma^{-1}(s)]$ is a non-degenerate interval in $\TI$, where we understand $\sigma^{-1}$ as the right-continuous inverse of $\sigma$.
This corresponds to an {\it instanton} at time $s$, meaning  that $F^\sigma$ takes all values
$\{F(t): t\in\sigma^{-1}(s)\}$
%%$F(t)$ for
%%$ t \in [\sigma^{-1}(s-), \sigma^{-1}(s)]$
``instantly'' at time $s$.

In the opposite direction, every function $f\in \D[I]$ can be interpreted as a turbofunction
\be \label{fplus}
   f^+:=(f\circ \varsigma)^{\varsigma},
\ee
where $\varsigma\in\Sigma$ is defined by $\varsigma(t):=t$.
Here and elsewhere the symbol $\circ$ stands for superposition of mappings.

In the following we will equip $\DDD$ with a relevant Skorokhod-type semi-distance and show that, after
the natural factorization, $\DDD$ can be interpreted as the completion of $(\D,\rho)$.

For continuous functions the situation is similar. If $f\in \C$, then $f^+\in \CCC$ and the natural factorization of $\CCC$ can be interpreted as the completion of $(\C,\rho)$.

\subsection{Distance and factorizations}

Let us define the Skorokhod
semi-distance\footnote{Semi-distance $\rho(\cdot,\cdot)$ is a
non-negative symmetric function satisfying triangle inequality and $\rho(x,x)=0$ but allowing
$\rho(x,y)=0$ for $x\not= y$.}
$\rho^+$ on turbofunctions in $\DDD$ by
\[
  \rho^+\left(F_1^{\sigma_1}, F_2^{\sigma_2} \right)
  := \inf_{\gamma\in \TGa}
  \left[
      \sup_{t\in\TI} |(F_1\circ\gamma)(t)-F_2(t)| + \sup_{t\in\TI} |(\sigma_1\circ\gamma)(t)-\sigma_2(t)|
  \right].
\]
For the restriction of $\rho^+$ on $\CCC$, we can replace the suprema in the above definition
by the maxima; however, this will play no role in the following.

It is easy to show that all properties of semi-distance are verified.
Indeed, the symmetry follows from
\begin{eqnarray*}
  &&
   \sup_{t\in\TI} |(F_1\circ\gamma)(t)-F_2(t)| + \sup_{t\in\TI} |(\sigma_1\circ\gamma)(t)-\sigma_2(t)|
\\
   &=&
    \sup_{\tau\in\TI} |F_1(\tau)-F_2(\gamma^{-1}(\tau))| + \sup_{\tau\in\TI}
    |\sigma_1(\tau)-\sigma_2(\gamma^{-1}(\tau))|,
\end{eqnarray*}
where we used the variable change $t=\gamma^{-1}(\tau)$.
To prove the triangle inequality
\be \label{triangular}
   \rho^+\left(F_1^{\sigma_1},F_3^{\sigma_3}\right) \le
   \rho^+\left(F_1^{\sigma_1},F_2^{\sigma_2}\right) +
   \rho^+\left(F_2^{\sigma_2},F_3^{\sigma_3}\right),
\ee
note that for all $t\in \TI$ and all  $\gamma_{12}, \gamma_{23}\in \TGa$ we have
\[
     |(F_1\circ\gamma_{12}\circ\gamma_{23} )(t)-F_3(t)|
     \le
     |(F_1\circ\gamma_{12}\circ\gamma_{23} )(t)-(F_2\circ \gamma_{23})(t)| + |(F_2\circ \gamma_{23})(t)-F_3(t)|.
\]
Taking supremum over $t \in \TI$ and using the variable change $\tau=\gamma_{23}(t)$
in the first supremum on the right-hand side, we get
\begin{eqnarray*}
  &&  \sup_{t\in\TI} |(F_1\circ\gamma_{12}\circ \gamma_{23} )(t)-F_3(t)|
\\
  &\le&
        \sup_{\tau\in\TI} |(F_1\circ\gamma_{12})(\tau)- F_2(\tau)| +
         \sup_{t\in\TI} |(F_2\circ \gamma_{23})(t)-F_3(t)|.
\end{eqnarray*}
Similarly,
\begin{eqnarray*}
  &&  \sup_{t\in\TI} |(\sigma_1\circ\gamma_{12}\circ \gamma_{23} )(t)-\sigma_3(t)|
\\
  &\le&
        \sup_{\tau\in\TI} |(\sigma_1\circ\gamma_{12})(\tau)- \sigma_2(\tau)| +
         \sup_{t\in\TI} |(\sigma_2\circ \gamma_{23})(t)-\sigma_3(t)|.
\end{eqnarray*}

Adding the two inequalities and subsequently optimizing over $\gamma_{12}, \gamma_{23}$
yields the triangle inequality \eqref{triangular}.
\bigskip

\begin{rem}
{\rm The space $\CCC$ is a closed subset of $(\DDD,\rho^+)$. The proof is almost the same as the one,
 recalled above, for the classical case $\C\subset\D$.
Indeed, if $F_n^{\sigma_n}\in \CCC$, $F^\sigma\in \DDD$, and
$\rho^+(F_n^{\sigma_n},F^\sigma)\to 0$, as $n\to\infty$, then there exists a sequence of functions
$\gamma_n\in \TGa$ such that
\[
   \lim_{n\to\infty} \sup_{t\in \TI} |(F_n\circ\gamma_n)(t)-F(t)|
   \le  \lim_{n\to\infty} \left[ \rho^+(F^\sigma,F_n^{\sigma_n}) + n^{-1}\right] =0.
\]
Therefore, $F$ is a uniform limit of continuous functions $F_n\circ\gamma_n$, hence, $F\in\C$ and
$F^\sigma \in \CCC$.
}
\end{rem}
\bigskip

The natural equivalence generated by the semi-distance $\rho^+$ is as follows:
\[
   F_1^{\sigma_1}\equiv F_2^{\sigma_2} \quad \textrm{iff } \quad \rho^+( F_1^{\sigma_1}, F_2^{\sigma_2})=0.
\]

It is worthwhile to compare this equivalence with the other one, which is more natural and very close
to $\equiv$ but still different from it. Namely, let
$F_1^{\sigma_1}\eqg F_2^{\sigma_2}$ iff there exists a homeomorphism $\gamma\in\TGa$ such that
$F_2=F_1\circ\gamma$ and $\sigma_2=\sigma_1\circ\gamma$.
Within a slightly different context, the equivalence $\eqg$ appeared in \cite[Section 15.7]{Whitt}.

\begin{prop} \label{p:diseq}
Let $F_1^{\sigma_1}, F_2^{\sigma_2} \in \DDD$. If $F_1^{\sigma_1}\eqg F_2^{\sigma_2}$,
then $F_1^{\sigma_1}\equiv F_2^{\sigma_2}$.
Conversely, if $F_1^{\sigma_1}\equiv F_2^{\sigma_2}$ and
$\sigma_1,\sigma_2$ are strictly increasing, then $F_1^{\sigma_1}\eqg F_2^{\sigma_2}$.
\end{prop}

\begin{proof} The first claim follows from the definitions of the semi-distance $\rho^+$ and the equivalences.
For the second claim, let us choose a sequence $\{\gamma_n\}_{n \in \N}$ in $ \TGa$ such that
\[
\lim_{n \to \infty}   \left[
        \sup_{t\in\TI} |(F_1\circ\gamma_n)(t)-F_2(t)|
      + \sup_{t\in\TI} |(\sigma_1\circ\gamma_n)(t)-\sigma_2(t)|
  \right] = 0.
\]
Then for an arbitrary $t'\in\TI$ it is true that $(\sigma_1\circ\gamma_n)(t')\to \sigma_2(t')$. Hence
$(F_1\circ\gamma_n)(t')\to (F_1 \circ \sigma_1^{-1} \circ \sigma_2)(t')$ provided that the function
$F_1 \circ \sigma_1^{-1}$ is continuous at $t'$. The latter assumption may cease to hold only for finitely
or countably many $t' \in \TI \setminus \{0,1\}$ since $F_1$ is a c\`adl\`ag function and $\sigma_1^{-1}$
is a continuous one.
On the other hand, we know that $(F_1\circ\gamma_n)(t')\to F_2(t')$.
Thus, $F_2(t')=F_1 \circ \sigma_1^{-1} \circ \sigma_2(t')$, and it follows that
$F_2=F_1 \circ \sigma_1^{-1} \circ \sigma_2$ since the functions on  both sides of this equality are c\`adl\`ag.
Since $\sigma_1^{-1} \circ \sigma_2 \in \TGa$, we have $ F_1^{\sigma_1} \eqg F_2^{\sigma_2}$, as required.
\end{proof}

\begin{cor} \label{c:diseq}If  $ F_1^{\sigma_1} \eqg F_2^{\sigma_2}$, then for every  $F_3^{\sigma_3}\in \DDD$ we have
\be   \label{eqn:distequiv}
   \rho^+\left(F_1^{\sigma_1}, F_3^{\sigma_3}\right)  = \rho^+\left(F_2^{\sigma_1}, F_3^{\sigma_3} \right).
\ee
\end{cor}
This claim immediately follows from Proposition \ref{p:diseq} and triangle inequality.
\bigskip

Somewhat surprisingly, the full converse in Proposition \ref{p:diseq} is not true, i.e.\ there exist
$ F_1^{\sigma_1}, F_2^{\sigma_2}$ such that $F_1^{\sigma_1} \equiv F_2^{\sigma_2}$
but $ F_1^{\sigma_1}  \eqg F_2^{\sigma_2}$ does not hold. Here is a counter-example.

\begin{example}
{\rm
Let $F_{\bf 1}(\cdot):=1$ and recall that $\varsigma\in\Sigma$ is defined as $\varsigma(t):=t$.
Then for every strictly increasing $\sigma\in\Sigma$ it is true
that $(F_{\bf 1})^\sigma \eqg (F_{\bf 1})^{\varsigma}$.
However, let us consider some $\sigma$ that is not strictly increasing and approximate it
uniformly with strictly increasing $\sigma_\delta\in \Sigma$, $0<\delta<1$,  so that
\[
    \sup_{t\in\TI} |\sigma_\delta(t)-\sigma(t)|\le \delta.
\]
This can be done by letting
\be \label{eqn:sigmadel}
  \sigma_\delta(t):= (1-\delta) \sigma(t) +\delta t, \qquad  t\in\TI.
\ee

We have
$\rho^+\left((F_{\bf 1})^{\sigma_\delta}, (F_{\bf 1})^{\sigma} \right)\le \delta$, which can be seen if we
let $\gamma\in \TGa$ be the identical homeomorphism of $\TI$ in the definition of $\rho^+$.
Furthermore,
\[
  \rho^+\left((F_{\bf 1})^{\varsigma}, (F_{\bf 1})^{\sigma} \right)
  \le  \rho^+\left((F_{\bf 1})^{\varsigma}, (F_{\bf 1})^{\sigma_\delta} \right)
  +  \rho^+\left((F_{\bf 1})^{\sigma_\delta}, (F_{\bf 1})^{\sigma} \right)
  \le 0 + \delta =\delta.
\]
Hence $\rho^+\left((F_{\bf 1})^{\varsigma}, (F_{\bf 1})^{\sigma} \right)=0$ by taking $\delta \searrow  0$.
In other words, $(F_{\bf 1})^{\varsigma} \equiv (F_{\bf 1})^{\sigma}$.

On the other hand, $(F_{\bf 1})^{\varsigma} \eqg (F_{\bf 1})^{\sigma}$ does not hold.
Indeed, assuming otherwise would imply that
there is a representation $\sigma=\varsigma\circ\gamma$ with some
$\gamma\in\TGa$. The right-hand side of this equality is strictly increasing,
while the left-hand side is not, which is a contradiction.
}
\end{example}
\medskip

The next proposition shows that both equivalences $\equiv$ and $\eqg$ respect the visualization
mapping.

\begin{prop} \label{p:eqviz}
Let $F_1^{\sigma_1}, F_2^{\sigma_2} \in \DDD$. Assume that
$\sigma_1,\sigma_2$ are strictly increasing and $F_1^{\sigma_1}\equiv F_2^{\sigma_2}$.
Then $\widehat{F_1^{\sigma_1}}=\widehat{F_2^{\sigma_2}}$.
\end{prop}

\begin{proof}
By Proposition \ref{p:diseq} we have $F_1^{\sigma_1}\eqg F_2^{\sigma_2}$, i.e.\ there exists a
$\gamma\in\TGa$ such that $F_2=F_1\circ\gamma$ and $\sigma_2=\sigma_1\circ\gamma$.
Then by visualization's definition we have
\[
   \widehat{F_2^{\sigma_2}}= F_2\circ \sigma_2^{-1}= (F_1\circ\gamma)\circ (\gamma^{-1}\circ\sigma_1^{-1})
   =  F_1\circ \sigma_1^{-1}=\widehat{F_1^{\sigma_1}}.
\]
\end{proof}

\section{Embedding of $\D$ into $\DDD$}

\begin{prop} \label{p:embedding}
The mapping $f\mapsto f^+$ defined by equation $\eqref{fplus}$ provides a dense
isometric embedding of $(\D,\rho)$ into $(\DDD,\rho^+)$. In particular, it isometrically
embeds $(\C,\rho)$ into  $(\CCC,\rho^+)$.
\end{prop}

\begin{proof} Let $f_1,f_2\in\D$. Then
\begin{eqnarray*}
 \rho^+(f_1^+,f_2^+)
%%   = \rho^+((f_1\circ \varsigma)^\varsigma,(f_2\circ\varsigma)^\varsigma)
  &=& \inf_{\gamma\in\TGa}\left[
  \sup_{t\in\TI}|(f_1\circ\varsigma\circ\gamma)(t)- (f_2\circ\varsigma)(t)|
  +   \sup_{t\in\TI}|(\varsigma\circ\gamma)(t)- \varsigma(t)|
  \right]
%%\\
%%  &=& \inf_{\gamma\in\TGa}\left[
%%  \sup_{s\in I}|(f_1\circ\varsigma\circ\gamma\circ\varsigma^{-1})(s)- f_2(s)|
%%  +   \sup_{s\in I}|(\varsigma\circ\gamma\circ\varsigma^{-1})(s)- s|
%%  \right]
\\
 &=& \inf_{\gamma'\in\Ga}\left[
  \sup_{s\in I}|(f_1\circ \gamma')(s)- f_2(s)|
  +   \sup_{s\in I}|\gamma'(s)- s|
  \right]
\\
  &=&  \rho(f_1,f_2),
\end{eqnarray*}
with $\gamma':= \varsigma\circ\gamma\circ\varsigma^{-1}$. This proves the isometry property.

Now we prove that the image $\{f^+:f\in\D\}$ is dense in $\DDD$.
To see this, let us consider some arbitrary $F^\sigma\in\DDD$ with
strictly increasing $\sigma$. Then we have
\begin{eqnarray} \nonumber
  \left[\widehat{F^\sigma}\right]^+ &=& \left[ F\circ\sigma^{-1}\right]^+
  = (F\circ\sigma^{-1}\circ\varsigma)^\varsigma
\\ \label{eqn:hatplus}
  &=&  (F\circ\sigma^{-1}\circ\varsigma)^{\sigma\circ(\sigma^{-1}\circ \varsigma)}
  \eqg F^\sigma.
\end{eqnarray}
We see that $F^\sigma$ is $\eqg$-equivalent to an element of the image.

Finally, let $F^\sigma$ be an arbitrary element of $\DDD$.
Consider its approximations $F^{\sigma_\delta}$ with strictly increasing $\sigma_\delta$
introduced in \eqref{eqn:sigmadel}.
Then we have
\[
   \rho^+(F^{\sigma},F^{\sigma_\delta}) \le \sup_{t\in\TI} |\sigma(t)-\sigma_\delta(t)|
  \le \delta \to 0,\qquad \textrm{as } \delta\to 0.
\]
By Corollary \ref{c:diseq} we have
\[
  \limsup_{\delta\to 0} \rho^+\big(F^{\sigma},  \big[\widehat{F^{\sigma_\delta}}\big]^+\big)
  =
   \limsup_{\delta\to 0} \rho^+(F^{\sigma}, F^{\sigma_\delta}) =0,
\]
which proves that $\{f^+: f\in\D\}$ is dense in $(\DDD,\rho^+)$.
\end{proof}
\bigskip

In Example \ref{example:triang}, we constructed a family of functions $\{g_\theta\}_{\theta>2}$
in $\D$ with the Cauchy property but not having a limit in $(\D,\rho)$, as $\theta\to\infty$.
We show now that the isometrically embedded family $\{g_\theta^+\}_{\theta>2}$
{\it has a limit} in $(\DDD,\rho^+)$.
This is a good hint at completeness
of $(\DDD,\rho^+)$.

\begin{example} \label{example:triang_continued}
{\rm  (Example \ref{example:triang} continued).

We use the notation from the mentioned example. Let
\[
   F(t):= g_4(t)=
   \begin{cases}
     0, & \mbox{if } 0\le t \le \frac 14, \\
     1-4|t-\frac 12|, & \mbox{if }  \frac 14\le t \le \frac 34, \\
     0, & \mbox{if }   \frac 34\le t \le 1.
   \end{cases}
\]
Introduce the time changes $\sigma_\theta\in\Sigma$ as piecewise
linear functions with nodes
\[
   \sigma_\theta(0):= 0; \quad \sigma_\theta(1/4):= \frac 12-\frac1\theta;
   \quad
   \sigma_\theta(3/4):= \frac 12+\frac1\theta; \quad \sigma_\theta(1):=1.
\]
Then we have
\[
    g_\theta(s)= F(\sigma_\theta^{-1})(s)= \widehat{F^{\sigma_\theta}}(s),
    \qquad s\in I,
\]
and the sequence $g_\theta^+ = [\widehat{F^{\sigma_\theta}}]^+ \eqg F^{\sigma_\theta}$
(where we used \eqref{eqn:hatplus} at the last step) has
a limit $F^\sigma$ in $\CCC$ (as $\theta\to\infty$)
with $\sigma\in\Sigma$ being the piecewise linear function with nodes
\[
    \sigma(0) := 0; \quad  \sigma(1/4):= \frac 12;
     \quad
    \sigma(3/4):= \frac 12; \quad \sigma(1):=1,
\]
because
\[
   \rho^+ \left(g_\theta^+, F^\sigma\right)
  = \rho^+ \left( F^{\sigma_\theta}, F^\sigma\right)
  \le \sup_{t\in \TI} |\sigma_\theta(t)-\sigma(t)| = \frac 1\theta.
\]
Since the limiting variable change $\sigma$ is not strictly increasing,
the limiting  turbofunction $F^\sigma$ has an instanton and can not
be interpreted as a usual function.
}
\end{example}

\section{Completeness of $(\DDD,\rho^+)$}
In the following proposition we essentially reach our goal by showing that $(\DDD,\rho^+)$
is complete.
\begin{prop} \label{p:complete}
Let $\{F_n^{\sigma_n}\}_{n \in \N}$ be a Cauchy sequence in $(\DDD,\rho^+)$. Then there exists
a limit $F^\sigma\in\DDD$ such that $\lim_{n \to \infty} \rho^+\big(F_n^{\sigma_n},F^\sigma\big)=0$.
Moreover, if $F_n^{\sigma_n}\in\CCC$ for all $n\in\N$, then $F^\sigma\in\CCC$.
\end{prop}

\begin{proof}
It suffices to show that $\{F_n^{\sigma_n}\}_{n \in \N}$ contains a converging subsequence.
Choose a subsequence $\{n_k\}_{k \in \N}$ in  $\N$ such that
$\rho^+(F_{n_k}^{\sigma_{n_k}},F_{n_{k+1}}^{\sigma_{n_{k+1}}} ) <1/2^k$.
By definition of the semi-metric $\rho^+$, for every $k \in \N$ there exists a $\gamma_k \in \TGa$
such that
\[
      \sup_{t\in\TI} |(F_{n_{k+1}}\circ\gamma_k)(t)-F_{n_k}(t)|
    + \sup_{t\in\TI} |(\sigma_{n_{k+1}}\circ\gamma_k)(t)-\sigma_{n_k}(t)| < \frac{1}{2^k}.
\]
Define $\lambda_k:= \gamma_{k-1} \circ \ldots \circ \gamma_1$ for integer $k \ge 2$ and let $\lambda_1$
be the identical homeomorphism of $\TI$. Since $\lambda_k \in \TGa$, the change of variables $t = \lambda_k(\tau)$
yields
\begin{multline*}
       \sup_{\tau\in\TI} |(F_{n_{k+1}}\circ\lambda_{k+1})(\tau)-(F_{n_k} \circ \lambda_k) (\tau)| \\
    + \sup_{\tau\in\TI} |(\sigma_{n_{k+1}}\circ\lambda_{k+1})(\tau)-(\sigma_{n_k} \circ \lambda_k)(\tau)|
    < \frac{1}{2^k}.
\end{multline*}

Hence the sequences $\{F_{n_k}\circ\lambda_k\}_{k \in \N}$ and $\{\sigma_{n_k}\circ\lambda_k\}_{k \in \N}$
are Cauchy sequences in $\D[\TI]$ (resp.\ $\C[\TI]$)  equipped with the uniform distance. Since these metric spaces
are complete (see \cite[Section 18]{Bil}), we can find $F\in \D[\TI]$ and $\sigma \in \C[\TI]$ such that
\[
    \lim_{k \to \infty} \left[ \sup_{\tau\in\TI} |(F_{n_k}\circ\lambda_k)(\tau)-F(\tau)|
    + \sup_{\tau\in\TI} |(\sigma_{n_k}\circ\lambda_k)(\tau)-\sigma(\tau)| \right ] = 0.
\]
We also have $\sigma \in \Sigma$, since $\sigma$ is non-decreasing as a pointwise
limit of non-decreasing functions $\sigma_{n_k}\circ\lambda_k$.
Finally, the above equality implies that
$\rho^+(F_{n_k}^{\sigma_{n_k}}, F^\sigma)\to 0$, as $k \to \infty$, as required.

Under the additional assumption $\{ F_n^{\sigma_n} \}_{n \in \N} \subset \CCC$, one may observe that
$\{F_{n_k} \circ \lambda_k\}_{k \in \N}\subset \C[\TI]$,
thus $F\in \C[\TI]$ and $F^\sigma\in\CCC$.
\end{proof}

\section{Concluding formalities}

Consider the quotient space $\D_\equiv^+:= \DDD / \! \equiv$.  Let
$\pi:\DDD\mapsto \D_\equiv^+$  denote the natural projection. The subspace $\CCC$
is factorized correctly because from $F_1^{\sigma_1}\equiv F_2^{\sigma_2}$
and $F_1^{\sigma_1}\in \CCC$ it follows that $F_2^{\sigma_2}\in \CCC$.
We denote  $\C_\equiv^+:= \CCC / \equiv$, so $\C_\equiv^+ =\pi(\CCC)$.

The {\it semi-distance} $\rho^+$ on $\DDD$ generates the {\it distance} on $\D_\equiv^+$.
We will also denote this distance by $\rho^+$, which should not cause
any confusion.

\begin{thm}
The spaces  $(\D_\equiv^+,\rho^+)$ and $(\C_\equiv^+,\rho^+)$ are isometrically isomorphic to the
completions of the respective spaces  $(\D,\rho)$ and $(\C,\rho)$.
\end{thm}

\begin{proof}
From Proposition \ref{p:complete} it follows that $(\D_\equiv^+,\rho^+)$
is a complete metric space.
From Proposition \ref{p:embedding} it follows that the mapping
\[
   f \mapsto \pi\big( f^+ \big)
\]
is an injective isometric embedding of $(\D,\rho)$ into $(\D_\equiv^+,\rho^+)$ and
its image $\left\{\pi\big( f^+ \big): f\in\D\right\}$ is dense in $(\D_\equiv^+,\rho^+)$.

Therefore, we can identify $(\D_\equiv^+,\rho^+)$ with the completion of  $(\D,\rho)$.

By the same arguments, we can identify $(\C_\equiv^+,\rho^+)$ with the completion of
$(\C,\rho)$.
\end{proof}

\begin{rem}
{\rm An equivalence class, i.e.\ an element of $\D_\equiv^+$, may be viewed as an element
of the completion, while the turbofunctions from this class may be viewed as its different
parametrizations. This is quite similar to a curve on a manifold having a multitude of parametrizations.}
\end{rem}

\section{Pointwise convergence}

Finally, we relate convergence in $(\DDD, \rho^+)$ to pointwise convergence.

Recall that every function $\sigma \in \Sigma$ is non-decreasing.
From this point on, let us agree to understand $\sigma^{-1}$  as the {\it right-continuous inverse},
namely
\begin{equation} \label{eqn:inverse}
       \sigma^{-1}(s):=\max \{t \in \TI: \sigma(t) \le s \}, \quad s \in [0,1].
\end{equation}
By doing so, we extend definition \eqref{eqn:visual} of the visualization mapping to the whole $\D$.
Then $ \widehat{F^\sigma} \in \D$ for any $F^\sigma \in \DDD$.

\begin{thm} \label{t:conv}
Suppose that $\lim_{n \to \infty} \rho^+(F_n^{\sigma_n}, F^\sigma)=0$ for some turbofunctions
$F^\sigma$, $F_1^{\sigma_1}, F_2^{\sigma_2}, \ldots$ in $\DDD$. Then
$\lim_{n \to \infty} \widehat{F_n^{\sigma_n}}(s) = \widehat{F^\sigma}(s)$ for $s =1$
and for every continuity point $s$ of $\sigma^{-1}$ such that $F$ is continuous at $\sigma^{-1}(s)$.
\end{thm}

We first give two corollaries.

\begin{cor}
Any Cauchy sequence in $(\D,\rho)$ converges pointwise to an element of $\D$ except, possibly,
for at most countable set of points in $[0,1)$.
\end{cor}

This follows from Theorem \ref{t:conv} combined with Propositions~\ref{p:embedding} and \ref{p:complete},
and using that
the functions $\sigma^{-1}$ and $F$ have at most countably many discontinuities.

\begin{rem}
{\rm We stress that the limiting element in $\D$ does not characterize completely
the limiting turbofunction $F^\sigma$ because it skips the instantons. In Example \ref{example:triang_continued}
the limiting function is zero but $F^\sigma$  is not degenerate.}
\end{rem}
\medskip

The next corollary asserts that Proposition~\ref{p:eqviz} still holds for the extended notion
of visualization.

\begin{cor}
   Let $F_1^{\sigma_1}, F_2^{\sigma_2} \in \DDD$ and $F_1^{\sigma_1}\equiv F_2^{\sigma_2}$.
   Then $\widehat{F_1^{\sigma_1}}=\widehat{F_2^{\sigma_2}}$.
\end{cor}

This follows from Theorem~\ref{t:conv} once we take $F_2^{\sigma_2}=F_3^{\sigma_3}=\ldots$ and
$F^\sigma = F_1^{\sigma_1}$ and use the fact that a c\`adl\`ag function on $[0,1]$
%% $\widehat{F_1^{\sigma_1}}$ and $\widehat{F_2^{\sigma_2}}$ are defined
is identified by its values on any dense set that includes $1$.
%%the set of convergence points $s$ in the theorem.
\medskip

\begin{proof}[{Proof of Theorem~\ref{t:conv}}]
By definition of the semi-metric $\rho^+$, there exist homeomorphisms $\gamma_1, \gamma_2, \ldots$ in $\TGa$
such that
\be \label{eqn:conv12}
    \lim_{n \to \infty} \left[ \sup_{t\in\TI} |(F_n\circ\gamma_n)(t)-F(t)|
       + \sup_{t\in\TI} |(\sigma_n \circ\gamma_n)(t)-\sigma(t)| \right ] = 0.
\ee

We first consider the case $s=1$. Since $\gamma_n(1)=1$ for all $n$, equality \eqref{eqn:conv12} implies $F_n(1)\to F(1)$.
This yields the claim for $s =1$, since $\sigma^{-1}(1)=\sigma_n^{-1}(1)=1$ for all $n \in \N$
by definition \eqref{eqn:inverse} of right-continuous
inverse\footnote{Notice that this argument does not apply to the opposite endpoint  $s=0$ of the time
interval because we do not have $\sigma^{-1}(0)=0$ in general. }
and we obtain
\[
    \lim_{n \to \infty} \widehat{F_n^{\sigma_n}}(1) = \lim_{n \to \infty} F_n\big(\sigma_n^{-1}(1)\big)
    = \lim_{n \to \infty} F_n(1) =F(1) = F\big(\sigma^{-1}(1)\big) = \widehat{F^{\sigma}}(1).
\]

It remains to consider the case when $\sigma^{-1}$
is continuous at $s$ and $F$ is continuous at $\sigma^{-1}(s)$.

We have
\begin{align*}
     \big |\widehat{F_n^{\sigma_n}}(s) - \widehat{F^\sigma}(s) \big |
     &=  \big | (F_n \circ \gamma_n) \circ (\gamma_n^{-1} \circ \sigma_n^{-1} )  (s) - (F \circ \sigma^{-1})(s) \big |
\\
     &\le \big |(F_n \circ \gamma_n) \big ((\gamma_n^{-1} \circ \sigma_n^{-1} )(s) \big)
        - F \big ( (\gamma_n^{-1} \circ \sigma_n^{-1}) (s) \big)  \big |
\\
     &\mathrel{\phantom{=}}  +  \big |  F \big ( (\gamma_n^{-1} \circ \sigma_n^{-1}) (s) \big)
       - F (\sigma^{-1}(s)) \big |,
\end{align*}
hence
\[
      \big |\widehat{F_n^{\sigma_n}}(s) - \widehat{F^\sigma}(s) \big |
      \le \sup_{t\in\TI} |(F_n\circ\gamma_n)(t)-F(t)|
         +  \big | F \big ( (\gamma_n^{-1} \circ \sigma_n^{-1}) (s) \big) - F (\sigma^{-1}(s)) \big | .
\]
By equality \eqref{eqn:conv12}, the first term on the right-hand side vanishes as $n \to \infty$,
and since $F$ is continuous at $\sigma^{-1}(s)$,
it remains to prove that
\[
   \lim_{n \to \infty } (\gamma_n^{-1} \circ \sigma_n^{-1}) (s) = \sigma^{-1}(s).
\]

The function $\sigma^{-1}$ is continuous at $s$ by theorem's assumption, hence there is a unique $t \in \TI$
such that $\sigma(t)=s$. Assume first that $0<t<1$. Let $t_1,t_2 \in \TI$ be such that
$t_1<t<t_2$. Then $\sigma(t_1)<\sigma(t)=s <\sigma(t_2)$.
By equality \eqref{eqn:conv12}, we have
\[
  \lim_{n \to \infty} (\sigma_n \circ  \gamma_n)(t_j)= \sigma(t_j), \quad j=1,2.
\]
Hence, for all $n$ large enough we have
\[
  (\sigma_n \circ  \gamma_n)(t_1) < s < (\sigma_n \circ  \gamma_n)(t_2).
\]
Since the function $\gamma_n^{-1} \circ \sigma_n^{-1}$ is non-decreasing, these inequalities yield
\be \label{eqn:t1t2}
   t_1 \le (\gamma_n^{-1} \circ \sigma_n^{-1})(s) \le t_2.
\ee
By taking the limits, we obtain
\[
   t_1 \le \liminf_{n\to\infty} (\gamma_n^{-1} \circ \sigma_n^{-1})(s)
       \le \limsup_{n\to\infty} (\gamma_n^{-1} \circ \sigma_n^{-1})(s)
   \le t_2.
\]
Finally, letting $t_1\nearrow t$ and $t_2\searrow t$, we get
\[
  \lim_{n\to\infty} (\gamma_n^{-1} \circ \sigma_n^{-1})(s)= t =\sigma^{-1}(s),
\]
as required.

If $t=0$ (resp.\ $t=1$), we can not choose $t_1$ (resp.\ $t_2$) as before. However,
inequality \eqref{eqn:t1t2} obviously holds with $t_1=t=0$ (resp.\ $t=t_2=1$),
and our conclusion for the limit follows as explained above.
\end{proof}

\section{Concluding remarks}

\qquad

{\bf 1.}\ For using the suggested explicit space construction in concrete questions, one
should first of all study the related compactness and tightness criteria.
This might be a line of subsequent research.
\medskip

{\bf 2.}\
%% The approach of the paper also works for
All results of the paper remain valid for c\`adl\`ag functions taking values in a complete
separable metric space, e.g.\ in $\R^d$ with $d>1$. The proofs are carried over without any 
change.
\medskip

{\bf 3.}\ Looking beyond probabilistic applications, one may mention {\it regulated functions} %%widely
used in solving differential and integral problems involving discontinuous solutions,
see e.g.~\cite{Cichon, Chis, Fran}. A regulated function is a function having finite
one-sided limits at every point, like a c\`adl\`ag function, but one-sided continuity is not assumed.
Therefore, Skorokhod space is a subspace of the space of regulated functions.
Introducing the regulated turbofunctions along the lines of the present article
seems perfectly possible, although we are not aware of any immediate application
fields for them.

\section*{Acknowledgements}
We are grateful to both anonymous referees for their insightful remarks.
The work of the first author was supported by RFBR-DFG grant 20-51-12004.
The work of the second author was supported in part by RFBR grant 19-01-00356.

\medskip

\textsc{Mikhail A. Lifshits, Saint-Petersburg State University, University Emb.\ 7/9, St Petersburg, 199034,  Russia}

{\it E-mail address}: mikhail@lifshits.org

\medskip

\textsc{Vladislav Vysotsky, University of Sussex, Pevensey 2 Building, Falmer Campus, Brighton BN1 9QH, United Kingdom} and \textsc{St Petersburg Department of Steklov Mathematical Institute, Fontanka~27,  St Petersburg, 191011, Russia}

{\it E-mail address}: v.vysotskiy@sussex.ac.uk


\begin{thebibliography}{99}

\bibitem{Bil}
Billingsley, P. Convergence of probability measures. \emph{John Wiley \& Sons, Inc.},
New York, 1968. % \MR{0233396}

\bibitem{Cichon}
Cicho\'n, K., Cicho\'n, M., Satco, B.
On regulated functions. \emph{Fasciculi Mathematici}
\textbf{60}, No.\,1, (2018), 37--57. %\MR{3846777}

\bibitem{Chis}
Chistyakov, V. V. Metric modular spaces: theory and applications. Ser.: Springer
Briefs in Math., \emph{Springer}, 2015.  %\MR{3410960}

\bibitem{Fran}
Fra\v{n}kov\'a, D. Regulated functions. \emph{Mathematica Bohemica} \textbf{116}, No.\,1,
 (1991), 20--59. %\MR{1100424}

\bibitem{Mogulskii1976}
Mogulskii, A. A. Large deviations for the trajectories of multidimensional random walks.
\emph{Theor. Probab. Appl.} \textbf{21},  No.\,2,  (1976) 300--315. %\MR{0420798}
%%Teor. Verojatnost. i Primenen., 1976, 21, No.~2, 309--323.

\bibitem{Mogulskii1993}
Mogulskii, A. A. Large deviations for processes with independent increments.
\emph{Ann. Probab.} \textbf{21}, No.\,1, (1993) 202--215. %\MR{1207223}

\bibitem{Skorokhod1956}
Skorokhod, A. V.  Limit theorems for stochastic processes. \emph{Theor. Probab. Appl.}
\textbf{1}, No.\,3,  (1956)  261--290. %\MR{0084897}

\bibitem{Skorokhod1957}
 Skorokhod, A. V. Limit theorems for stochastic processes with independent increments.
 \emph{Theor. Probab. Appl.}  \textbf{2}, No.2, (1957), 138--171. %\MR{0094842}

\bibitem{Whitt}
Whitt, W.: Stochastic-process limits. \emph{Springer-Verlag}, New York, 2002.  %\MR{1876437}

\end{thebibliography}
\end{document}